\documentclass[reqno,11pt]{amsart}

\usepackage[utf8]{inputenc}
\usepackage[english]{babel}

\usepackage{verbatim} 
\usepackage[normalem]{ulem} 

\usepackage{fullpage}
\usepackage{setspace}
\usepackage{amsmath}
\usepackage{amssymb}
\usepackage{graphicx}
\usepackage{datetime}
\usepackage{enumitem}
\usepackage{mathtools}
\usepackage{array}
\usepackage{bm}
\usepackage{caption}
\usepackage{subcaption}
\usepackage{empheq}
\usepackage[mathscr]{eucal}

\newtheorem{theorem}              {Theorem}
\newtheorem{lemma}      [theorem] {Lemma}

\newtheorem{corollary}  [theorem] {Corollary}
\newtheorem{conjecture} [theorem] {Conjecture}

\theoremstyle{definition}

\newtheorem{definition} [theorem] {Definition}

\def\alabel{\upshape({\itshape \alph*\,})}

\usepackage{tikz}  
\usetikzlibrary{shapes}
\usetikzlibrary{arrows}
\usetikzlibrary{calc}
\usetikzlibrary{decorations.markings}
\usetikzlibrary{arrows.meta}
\usetikzlibrary{shapes.misc}   
\usetikzlibrary{positioning}
\usetikzlibrary{fit}
\usetikzlibrary{external} 
\tikzset{non code/.style={circle, draw=black}}
\tikzset{code/.style={circle, fill=black}}
\tikzset{vtx/.style={circle, draw=black, thick}}

\usepackage{dsfont} 
\usepackage[babel]{microtype}

\usepackage{xcolor} 
\usepackage{hyperref}

\usepackage[T1]{fontenc}

\makeatletter
\def\moverlay{\mathpalette\mov@rlay}
\def\mov@rlay#1#2{\leavevmode\vtop{   \baselineskip\z@skip \lineskiplimit-\maxdimen
   \ialign{\hfil$\m@th#1##$\hfil\cr#2\crcr}}}
\newcommand{\charfusion}[3][\mathord]{
    #1{\ifx#1\mathop\vphantom{#2}\fi
        \mathpalette\mov@rlay{#2\cr#3}
      }
    \ifx#1\mathop\expandafter\displaylimits\fi}
\makeatother

\newcommand{\quot}[2]{\mathchoice%
{\left.\raisebox{.1em}{$\displaystyle{#1}$}\kern-1pt/\raisebox{-.2em}{$\displaystyle{#2}$}\right.}
{\left.\raisebox{.1em}{${#1}$}\kern-1pt/\raisebox{-.2em}{${#2}$}\right.}
{\left.\raisebox{.1em}{$\scriptstyle{#1}$}\kern-1pt/\raisebox{-.2em}{$\scriptstyle{#2}$}\right.}
{\left.\raisebox{.1em}{$\scriptscriptstyle{#1}$}\kern-1pt/\raisebox{-.2em}{$\scriptscriptstyle{#2}$}\right.}
}

\DeclareFontFamily{U}  {MnSymbolC}{}
\DeclareSymbolFont{MnSyC}         {U}  {MnSymbolC}{m}{n}
\DeclareFontShape{U}{MnSymbolC}{m}{n}{
    <-6>  MnSymbolC5
   <6-7>  MnSymbolC6
   <7-8>  MnSymbolC7
   <8-9>  MnSymbolC8
   <9-10> MnSymbolC9
  <10-12> MnSymbolC10
  <12->   MnSymbolC12}{}
\DeclareMathSymbol{\powerset}{\mathord}{MnSyC}{180}

\let\epsilon\varepsilon

\let\hat\widehat

\usepackage{lineno}  
 \newcommand*\patchAmsMathEnvironmentForLineno[1]{%
 \expandafter\let\csname old#1\expandafter\endcsname\csname #1\endcsname
 \expandafter\let\csname oldend#1\expandafter\endcsname\csname end#1\endcsname
 \renewenvironment{#1}%
 {\linenomath\csname old#1\endcsname}%
 {\csname oldend#1\endcsname\endlinenomath}}%
 \newcommand*\patchBothAmsMathEnvironmentsForLineno[1]{%
 \patchAmsMathEnvironmentForLineno{#1}%
 \patchAmsMathEnvironmentForLineno{#1*}}%
 \AtBeginDocument{%
 \patchBothAmsMathEnvironmentsForLineno{equation}%
 \patchBothAmsMathEnvironmentsForLineno{align}%
 \patchBothAmsMathEnvironmentsForLineno{flalign}%
 \patchBothAmsMathEnvironmentsForLineno{alignat}%
 \patchBothAmsMathEnvironmentsForLineno{gather}%
 \patchBothAmsMathEnvironmentsForLineno{multline}%
 }

\allowdisplaybreaks[2]

\newcommand{\lds}{\textup{\textsc{lds}}}
\newcommand{\dist}{\textrm{dist}}
\newcommand{\MinDen}{\textup{\textsc{Min-density}}}

\begin{document}
\onehalfspace
\footskip=28pt

\title{Optimal and quasi-optimal locating-dominating densities in the infinite hexagonal
    grid with a finite number of rows}

\author{Arthur C. Gomes}
\address{Institute of Mathematics, Statistics and Computer Science - 
    University of S\~ao Paulo, Brazil} 
\email{arthurcgomes@ime.usp.br}

\author{Yoshiko Wakabayashi}
\address{Institute of Mathematics, Statistics and Computer Science -
    University of S\~ao Paulo, Brazil}
\email{yw@ime.usp.br}

\begin{abstract}
    A set of vertices $S$ of a graph $G$ is locating-dominating if $S$
     is dominating and, for each pair of distinct vertices not in
     $S$, their neighborhoods in $S$ are distinct.
    We present results on the minimum density of such sets in the
    infinite hexagonal grid with a finite number of rows~$k$, also
    known as the hexagonal strip of width $k$, which we denote by
    $H_k$.  For each $k\geq 2$, we present either an optimal solution
    or a quasi-optimal solution for $H_k$ that is within $1.3\%$ of
    the optimum.  
    We describe an exact exponential-time algorithm for fixed k, which
    we implemented to find optimal solutions for~$k \leq 5$.
    As the infinite grid $H_{k}$ always admits a periodic optimal
    solution, to deal with larger values of $k$, we present an integer
    linear program that finds an optimal periodic solution for~$H_{k}$
    for each fixed period.  This program yields high-quality feasible
    solutions for~$H_7$ and~$H_8$, which we then combine with an
    optimal solution for $H_3$ to obtain quasi-optimal solutions for
    all~$k\geq 6$. All these solutions admit a very short description.

    \medskip
    
\end{abstract}


\shortdate
\yyyymmdddate
\def\today{\number\year/\number\month/\number\day}
\settimeformat{ampmtime}
\date{\today, \currenttime}


\maketitle

\section{Introduction}
\label{sec:introduction}

Let $G = (V, E)$ be a connected graph, and let $\text{dist}(u, v)$
denote the distance between two vertices~$u, v \in V$. We say that
vertices at distance~1 are \emph{neighbors}. The \emph{(open)
neighborhood} of a vertex $v \in V$, denoted $N(v)$, is defined as
$N(v) \coloneqq \{u \in V : \text{dist}(v, u) = 1\}$, while the
\emph{closed neighborhood} of~$v$, denoted $N[v]$, is defined as
$N[v] \coloneqq N(v) \cup \{v\}$.  A \emph{dominating set} of
$G$ is a set~$S \subseteq V$ such that each vertex $u \in V \setminus S$
has a neighbor in $S$. A dominating set $S$ of $G$ is a
\emph{locating-dominating set} (\lds{}, for short) if $S$ 
has the additional \emph{locating property}: each
vertex $v \in V \setminus S$ can be uniquely distinguished by its
neighborhood in $S$.  Formally, this means that for each pair of
distinct vertices $u, v \in V \setminus S$, we must have
$N(u) \cap S \neq N(v) \cap S$.  Locating-dominating sets were
introduced by Slater~\cite{Slater75,Slater87}. More recently, this concept and
several of its close variants have been studied on regular infinite
grids. Here, we focus on the problem of finding minimum-density
locating-dominating sets in infinite hexagonal grids of finite height
(also known as infinite hexagonal strips of finite width).

Beyond their combinatorial interest, locating-dominating sets can be used 
to model fault diagnosis in multiprocessor networks and intrusion detection in 
buildings. 
For the latter, the vertices of a graph may represent rooms while the edges
represent adjacencies between them. 
Each detector distinguishes whether an intruder is located in its own room or in 
one of the neighboring rooms. 
The properties of an \lds{} guarantee that, from the collection of detector 
reports, the location of any single intruder can be uniquely determined.
Infinite regular grids provide a natural setting for studying the asymptotic 
behavior of locating-dominating sets. 
In particular, the finite-height grids may be viewed as idealized models of long, 
narrow networks.

Locating-dominating sets have been well-studied for finite graphs,
where the primary objective is to find an {\lds} of minimum
cardinality.
This problem is known to be NP-hard, even when restricted to specific
graph classes such as bipartite graphs~\cite{CharonHL03}, interval
graphs~\cite{FoucaudMNPV16}, and subcubic planar bipartite
graphs~\cite{Foucaud15}. Moreover, it is log-APX-hard for general
finite graphs and remains so for several graph
classes~\cite{Suomela07, Foucaud15}. In contrast, research for
infinite graphs has focused  predominantly on regular grids or strips of
finite width, where the goal generalizes to finding an {\lds} of
minimum density.  Recently, it has been proved that this 
 problem is NP-hard on infinite $\mathbb{Z}$-periodic
  graphs with a finite period~\cite{GomesW26b}.
  
To define density of a set on an infinite graph, let us first
establish some notation. For other concepts not defined here, the
reader is referred to~\cite{BondyM08}.
  
Let $v$ be a vertex in a graph $G=(V,E)$, and  $r \geq 1$, a natural number.
The \emph{$r$-open neighborhood} of $v$ in $G$ is defined as the set
$N_r(v) \coloneqq \{w \in V(G) : 0 < \dist(v, w) \leq r\}$.

The \emph{density} of a set $S$ in $V$, denoted
$d(S, G)$, is defined as
\vspace{-2mm}
\begin{align*}
    d(S, G) \coloneqq \inf \{d_w(S, G) : w \in V\}, \;\; \textrm{where} ~~ d_w(S, G) 
    \coloneqq \limsup_{r\to\infty} |N_r(w) \cap S| / |N_r(w)|.
\end{align*}

  The \emph{minimum density of an \lds{}} of a graph $G$, denoted
  $d^*(G)$, is defined as
  \[
   d^*(G) \coloneqq \inf \{d(S, G) : S \textrm{ is an \lds{} of } G\}.
  \]

This work focuses on the infinite hexagonal grid of finite height. 
This grid, together with three other regular grids, defined below, is among the 
most studied infinite regular grids for which research of this nature has been done.

Let $\mathcal{G}_S$, $\mathcal{G}_T$, $\mathcal{G}_K$ and $\mathcal{G}_H$ denote, 
respectively, the infinite \emph{square}, \emph{triangular}, \emph{king} and 
\emph{hexagonal} grids (see Figure~\ref{fig:infinte_grids}).
Each of these grids has vertex set $\mathbb{Z} \times \mathbb{Z}$, and edge
sets defined as follows:




    


\begin{itemize}
    \item Square grid: $\quad E(\mathcal{G}_S) \coloneqq
        \{\{u, v\} : u - v \in \{(\pm 1, 0), (0, \pm 1)\}\}$;

    \item Triangular grid: $\quad E(\mathcal{G}_T) \coloneqq 
        E(\mathcal{G}_S) \cup \{\{u, v\} : u - v \in \{ (1, 1), (-1, -1) \} \}$;

    \item King grid: $\quad E(\mathcal{G}_K) \coloneqq
            E(\mathcal{G}_T) \cup \{\{u, v\} : u - v \in \{ (1, -1), (-1, 1) \} \};$

    \item Hexagonal grid: $\quad  E(\mathcal{G}_H) \coloneqq
      \{\{u, v\} : u = (i, j) \textrm{ and } u - v \in \{(\pm
      1, 0), (0, (-1)^{i+j-1})\}\}.$
    
  \end{itemize}


\begin{figure}[ht]
    \centering
    \begin{subfigure}{0.22\textwidth}
        \centering
        \includegraphics{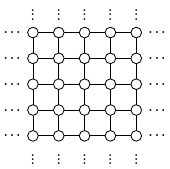}
        \caption{Square grid $\mathcal{G}_S$}
    \end{subfigure}\hfill
    \begin{subfigure}{0.22\textwidth}
        \centering
        \includegraphics{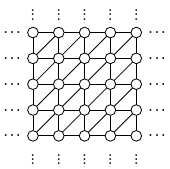}
        \caption{Triangular grid $\mathcal{G}_T$}
    \end{subfigure}\hfill
    \begin{subfigure}{0.22\textwidth}
        \centering
        \includegraphics{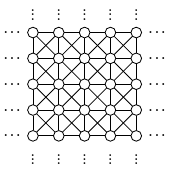}
        \caption{King grid $\mathcal{G}_K$}
    \end{subfigure}\hfill
    \begin{subfigure}{0.22\textwidth}
        \centering
        \includegraphics{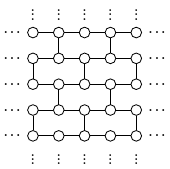}
        \caption{Hexagonal grid $\mathcal{G}_H$}
    \end{subfigure}
    \caption{Examples of regular grids: 4-regular, 6-regular, 
        8-regular and 3-regular.}
    \label{fig:infinte_grids}
\end{figure}

Slater~\cite{Slater02} proved that $d^*(\mathcal{G}_S) = 3/10$.
Honkala and Laihonen~\cite{HonkalaL06} showed that 
$d^*(\mathcal{G}_K) = 1/5$ and $d^*(\mathcal{G}_H) = 1/3$. Furthermore,
Honkala~\cite{Honkala06} showed that  $d^*(\mathcal{G}_T) = 13/57$.

Studies on infinite subgraphs of the above mentioned grids have also
been done, when the number of rows (also called height) is finite.  We
denote by $S_k$, $T_k$, $K_k$ and $H_k$ the square, triangular, king
and hexagonal grids with $k$ rows, respectively. Each of them is a
subgraph of its corresponding infinite grid induced by the vertex set
$[k] \times \mathbb{Z}$, where $[k] \coloneqq \{1, \ldots, k\}$. They
are also called \emph{strips} (of fixed width).
Table~\ref{tab:min-densities} shows the results that are known for
these grids.

\renewcommand{\thefootnote}{\fnsymbol{footnote}}
\begin{table}
    \centering
    \begin{tabular}{|c|c|c|c|c|}
        \hline
        $k$ & $S_k$ & $T_k$ & $K_k$ & $H_k$ \\
        \hline \hline
        $2$ & $3/8$~~\cite{BouznifDMP19} & $1/3$~~~\cite{BouznifDMP19}& $1/2$~~~\cite{BouznifDMP19} & $3/8$ \\
        \hline
        $3$ & $1/3$~~\cite{BouznifDMP19}& $3/10$~~\cite{BouznifDMP19}& $4/15$~~\cite{BouznifDMP19}& $1/3$ \\
        \hline
        $4$ & $1/3$~~\cite{GomesW25} & $5/18$~~\cite{GomesW25} & $1/4$~~~\cite{GomesW25} & $11/32$ \\
        \hline
        $5$ & $1/3$~~\cite{GomesW25} & $4/15$~~\cite{GomesW25} & $6/25$~~\cite{GomesW25} & $12/35$ \\
        \hline
        $6$ & $1/3$~~\cite{GomesW25} &  $11/42$\footnotemark[2]  &  $1/4$\footnotemark[2]  & $1/3$ \\
        \hline
    \end{tabular}
    \caption{Minimum densities LDSs in the grids $S_k$, $T_k$, $K_k$, ${H}_k$ 
        for $k\leq 6$.}
    \label{tab:min-densities}
\end{table}

\footnotetext[2]{These two values were computed by the authors and are included here for completeness.
Details will appear in a forthcoming paper.}
\renewcommand{\thefootnote}{\arabic{footnote}}


A concept related to \lds{} is that of an \emph{identifying code} of a
graph $G$, which is a dominating set~$C$ such that
$N[u] \cap C \neq N[v] \cap C$ for every pair of distinct vertices
$u, v \in V$.  In 2018, Jiang~\cite{Jiang18} introduced an
exponential-time algorithm to find minimum-density identifying codes
in $S_k$ using an idea based on the construction of a configuration
digraph.  Later, Bouznif, Darley, Moncel and
Preissmann~\cite{BouznifDMP19} determined, using a similar technique,
the minimum density of locating-dominating sets, identifying codes,
and locating-total-dominating sets for $S_k$, $T_k$ and $K_k$, when
$k \leq 3$ (and up to $k \leq 4$ for identifying codes in $S_k$).
Previously, Junnila~\cite{Junnila15} had determined the minimum
density of a locating-total dominating set in~$S_3$ using the
discharging method.  More recently, Sampaio, Sobral and
Wakabayashi~\cite{SampaioSW24} found minimum-density identifying codes
for the hexagonal grid $H_k$ with $k \leq 5$.  When $k = 1$, these
grids are all isomorphic to the infinite path, for which the minimum
density of an \lds{} is known to be $2/5$~\cite{BertrandCHL04}.

\smallskip

\noindent\textbf{Conventions and organization of the paper.}

\noindent In this work, when we mention a graph $G$, if its vertex set
and its edge set are not explicitly given, we adopt the convention
that they are $V(G)$ and $E(G)$, respectively.  The main problem of
this work will be referred to as the {\MinDen} {\lds} problem. Its
input consists of an infinite graph, and its goal is to find a
minimum-density {\lds}. Throughout this work we assume that 
$k$ is a  positive integer; and $H_k$ is the
infinite hexagonal grid with $k$ rows.

In \textbf{Section~\ref{sec:min-density-hex-grid}}, we present an
exact algorithm that we implemented to solve the {\MinDen} {\lds}
problem for~$H_k$, that is based on the configuration digraph
approach~\cite{Jiang18}.  We provide some implementation details and
exhibit the optimal {\lds} solutions we found for the grids $H_2$,
$H_3$, $H_4$ and $H_5$.
Next, in \textbf{Section~\ref{sec:ilp}} we present an integer linear
program that finds, for a fixed integer~$k$, an optimal periodic feasible
\lds{} solution for $H_k$ with a fixed (even) period $\ell\geq 6$.  We
exhibit the feasible solutions for $H_7$ and $H_8$ obtained with this
approach, and show that their densities are very close to the optimum.
Finally, in \textbf{Section~\ref{sec:stack-hex-grid}} we present a
constructive technique to obtain new feasible periodic solutions by
combining feasible periodic solutions.  Using this construction, we
show how to explicitly obtain an optimal or quasi-optimal {\lds}
solution for~$H_k$ for all $k\geq 6$.  Moreover, as these solutions
are periodic, their descriptions are very short.  We also
  formulate a conjecture regarding an exact formula for the minimum
  density of locating-dominating sets in $H_k$.

\section{An exact algorithm to find minimum-density locating-dominating sets in 
    \texorpdfstring{$H_k$}{Hk}}
\label{sec:min-density-hex-grid}

In this section we present an exact algorithm, called $\mathbf{A_k}$,
for every fixed $k\geq 2$, 
for the \MinDen{} \lds{} problem in $H_k$.
This algorithm uses an approach based on the construction of a
\emph{configuration digraph} that contains all feasible periodic
\lds{} solutions. This approach is very general and can be
  adapted to many problems of similar nature.  As we may choose how to
define the vertices and arcs of the configuration digraph, after
testing some different choices, we fixed those that we found to be
more appropriate. Our description corresponds to these choices.

We included the complete description of the algorithm to make
  the paper self-contained, and also to prove the statement regarding
  the relation of an optimal periodic {\lds} solution with the
  minimization problem defined on the configuration digraph.

Let $\ell\geq 4$ be a natural number. We define an \emph{$\ell$-bar} of the hexagonal
grid $H_k$ with $k \geq 1$ rows as the subgraph induced by the vertices in the set
$[k] \times \{j_1, \ldots, j_{\ell}\}$, where $[k] \coloneqq \{1, 2, \ldots, k\}$, and
$j_1, \ldots, j_{\ell}$ represent $\ell$ consecutive columns of $H_k$.

Let $R$ be an $\ell$-bar of $H_k$, and let $S \subseteq V(R)$.
We say that $S$ is a \emph{barcode} of $R$ if $S$ distinguishes the vertices 
in the ``interior'' of the $\ell$-bar~$R$. That is, $S$ is a barcode if the 
$(\ell-2)$-bar, say $R'$, indexed by the columns $2$ through $\ell - 1$ has
the following properties: 
\begin{enumerate}[label=(\roman*)]
    \item for each vertex $v$ in $V(R') \setminus S$, we have $N(v) \cap S \neq \emptyset$; and 
    \item for each pair of distinct vertices $u$, $v$ in $V(R') \setminus S$, 
        we have $N(u) \cap S \neq N(v) \cap S$. 
\end{enumerate}

In other words, the subset $S$ satisfies the properties of an {\lds} for all vertices
located in the interior columns of the $\ell$-bar $R$. {Although a barcode $S$ is a
subset of vertices of an $\ell$-bar $R$, for simplicity, we may refer to the ``columns of
the barcode $S$'', with the understanding that we refer to the
columns of~$R$ on which~$S$ is defined.

We define the weighted \textbf{\lds{}-configuration digraph} $(G_k, w)$ 
in the following way.

\begin{enumerate}[label=(\roman*)]
    \item \textbf{Vertices:} each vertex $v_B$ of $G_k$ corresponds to a distinct
        barcode $B$ of a $4$-bar.

      \item \textbf{Arcs:} for each pair of two distinct vertices
        $v_B$ and $v_{B'}$, there is an arc from $v_B$ to $v_{B'}$ if
        the \emph{last two columns} of barcode $B$ are identical to
        the \emph{first two columns} of barcode~$B'$, and the overlap
        of the two identical columns of~$B$ and~$B'$ forms a barcode
        of the resulting $6$-bar.

    \item \textbf{Weights:} the weight $w_a$  of an arc $a=(v_B, v_{B'})$ is 
        equal to the number of vertices in $S$ in the \emph{last two columns} 
        of the barcode $B'$.
\end{enumerate}

To understand better the idea behind the construction of the
digraph $G_k$, let us call attention to the meaning of the arcs of
this digraph and convince ourselves that $G_k$ captures all feasible
periodic {\lds}s of $H_k$. See Figure~\ref{fig:cycle-Gk}. 
When we add an arc in $G_k$ linking two
vertices $v_B$ and $v_{B'}$ of~$G_k$ (when the overlap of the two
sets of two columns defined in (ii) happens), we mean that these two
overlapped columns may be part of a feasible {\lds} of $H_k$. Thus,
in $G_k$, a directed path with~$t$ arcs indicates that the sequence
of the $2t$ overlapped columns corresponding to the $t$ arcs of this
path may also be part of a feasible {\lds} of~$H_k$. If, in addition
to such a directed path, $G_k$ has an arc linking the end vertex of
this path to its initial vertex, we get a directed cycle (with $t+1$
arcs); moreover, this means that if we repeat the $2t$ columns side
by side, we get a periodic {\lds} of $H_k$ (with period $2t$).  The
converse also holds: whenever we have a periodic {\lds} of $H_k$
with period $2t$, we also have a cycle in $G_k$ with $t$ arcs. We
prove these and further assertions in the next theorem.

Henceforth, a cycle in $G_k$ means a directed cycle (and it is
represented by its sequence of vertices, without repetitions). For a
cycle $C$, the \emph{mean weight} of $C$, or simply, \emph{mean} of
$C$, is defined as the sum of the weights of the arcs in $C$ ($w(C)$)
divided by the number of arcs in $C$. The relation between a 
 minimum mean cycle of $(G_k, w)$  and a minimum-density {\lds} of 
 $H_k$ is explained in what follows.

 \begin{figure}
    \includegraphics{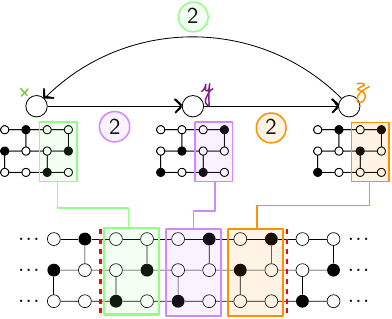}
    \caption{Example of a cycle in the configuration digraph $G_3$.}
    \label{fig:cycle-Gk}
 \end{figure}

 \medskip

 The algorithm $\mathbf{A_k}$
 \begin{enumerate}
   \item Construct the  weighted \lds{}-configuration digraph
     $(G_k, w)$;
   \item Find a minimum mean cycle $C$ in $(G_k, w)$;
   \item Return the {\lds} solution of $H_k$  that corresponds to $C$.
   \end{enumerate}

   More details on finding a mean cycle $C$ in $(G_k, w)$ is given
   later, and how to build from $C$ an {\lds} of $H_k$ is mentioned in
   Theorem~\ref{thm:min_density}.

   \medskip

\noindent \textbf{Periodic solutions and patterns.}\\
Given a feasible {\lds} $S$ of $H_k$, we say that $S$ is a
\emph{periodic solution} of $H_k$ if there is a natural number~$p$
such that we may partition $H_k$ into $p$-bars, such that the
subgraphs induced by $S$ in these $p$-bars are all identical, that is,
the solution $S$ can be obtained by repeating side by side (without
overlap) the solution in one of the $p$-bars.  We use the term
\emph{pattern} or \emph{periodic pattern} to refer to one of the
$p$-bars and the solution defined by $S$ on the corresponding $p$-bar.
In this case, we say that the solution $S$ has a pattern of
\emph{period} $p$.  Conversely, if $B$ is a $p$-bar of $H_k$ together
with some set of selected vertices on it, then we may also say that
$B$ is a pattern of $H_k$ if an infinite repetition of~$B$ side by
side yields a feasible periodic {\lds} of $H_k$.  Note that, for a
fixed periodic solution $S$ and a fixed~$p$, there may be different
ways to partition $H_k$ into $p$-bars, and each different partition
may define a different pattern. We say that a pattern is
\emph{optimal} if its repetition yields an {\lds} of minimum density.
In the forthcoming figures showing periodic solutions, the $p$-bar
corresponding to a pattern is indicated between two consecutive red
vertical dashed lines.


\medskip

\noindent\textbf{Remark.} Because of the structure of the hexagonal grid,
  and because in our implementation the vertices of $G_k$ correspond
  to barcodes of $4$-bars, and the arcs of $G_k$ correspond to an
  overlap of two columns, in $G_k$ we may only find periodic {\lds}s
  with even period.  However, we should note that this does not
  exclude any possible minimum density solution, since for any
  odd-length pattern, we can have an even-length pattern formed by
  taking two copies of such an odd-length pattern.  We could have
  defined that the vertices of $G_k$ are barcodes of $\ell$-bar with
  $\ell > 4$. From our experience, the value $\ell=4$ seemed simpler
  and more appropriate. We note that we allow loops (cycles of
  length~$1$) in $G_k$.

  \medskip

  The next theorem establishes a relationship between mean weight of
  cycles of the configuration digraph $(G_k, w)$ and densities of
  periodic locating-dominating sets of $H_k$. This theorem is
  important to guarantee that the solution found by the algorithm on
  the configuration digraph indeed finds an optimal {\lds} solution.

  As we refer to periodic solutions $S$ in $H_k$, and the density of
  $S$ in $H_k$, the following concept is important.




\begin{definition}
  An infinite connected graph $\hat{G}$ has the \emph{Slow Growth property}
  if it has a vertex $s$ such that
  $ \lim_{r \to \infty} {|N_{r+1}[s]|}/{|N_{r}[s]|} = 1$.
\end{definition}

As $H_k$ has the Slow Growth property, the following result proved by
Sampaio et al.~\cite{SampaioSW24} will be useful.

\begin{lemma}[Sampaio et al.~\cite{SampaioSW24}] \label{lem:SG-property} 
    Let $\hat{G}$ be an infinite connected graph with bounded maximum
    degree that satisfies the Slow Growth property. Let $\ell, c, c', \delta$ be
    positive integers, and let $S$ be a subset of
    $V(\hat{G})$. If $\hat{G}$ can be
    partitioned into finite sets $V_1,V_2, \ldots$  of size $\ell$ such that
    $c\leq |V_i\cap S|\leq c'$ for $i\in \mathbb{Z}$, and the distance between
    any two vertices of $V_i$ is at most $\delta$, then the density of $S$
    in $\hat{G}$  satisfies $c/\ell \leq  d(S,\hat{G}) \leq c'/\ell$.
\end{lemma}


\begin{theorem}
\label{thm:min_density}
Let $H_k$ be the infinite hexagonal grid with $k\geq 2$ rows, and let
$(G_k, w)$ be the weighted \lds{}-configuration digraph as defined previously. 
Then the following holds. 
  \begin{enumerate}[noitemsep,topsep=0.5mm,label=\alabel]
  \item The digraph $G_k$ has a cycle $C$ of length~$p\geq 1$ if, and
    only if, $H_k$ has a periodic \lds{} $S$ with a pattern $P$ of
    period $2p$.  Furthermore, the density of $S$ in $H_k$ and the
    mean of cycle $C$, $\widetilde{w}(C)$, satisfy the relation
    $d(S,H_k)=\widetilde{w}(C)/2k$. Thus, a minimum mean cycle in
    $G_k$ yields a minimum-density periodic \lds{} of $H_k$.

  \item As $H_k$ always admits an optimal \lds{} solution that is
    periodic, we have that $d^*(H_k) = \widetilde{w}(C^*)/2k$, where
    $C^*$ is a minimum mean cycle in $G_k$.
\end{enumerate} 
\end{theorem}


\begin{proof}
    Let $H_k$, $(G_k,w)$ and $k$ be as in the statement of the theorem.
    Let $\mathcal{C} = (v_{B_1}, \ldots, v_{B_p})$ be a cycle of length
    $p$ in $G_k$. Then, each $B_i$ is a barcode of a $4$-bar whose last
    two columns is identical to the first two columns of $B_{i+1}$, for
    $i =1,2,\ldots, p-1$. Since $C$ is a cycle,  the last two columns of
    the barcode $B_p$ is identical to the two first columns of the
    barcode $B_1$.

    Let $\mathcal{B} = B_1B_2\ldots B_pB_{p+1}$ be the $(2p+4$)-bar obtained
    by overlapping the last 2 columns of $B_j$ with the first 2
    columns of $B_{j+1}$, for $j=1,2,\ldots,p$, and where
    $B_{p+1}\equiv B_1$. Then $\mathcal{B}$ is clearly a
    barcode. Let~$P$ be the barcode consisting of the central $2p$
    columns of the barcode~$\mathcal{B}$, and let $S$ be a periodic
    solution of $H_k$ that is obtained by repeating side by side the
    barcode~$P$. Thus, by the definition of the weight of the arcs in
    $G_k$, we have that $w(C)$, the weight of $C$, is precisely the
    number of vertices of $S$ in the barcode~$P$.  Hence, by
      Lemma~\ref{lem:SG-property}, the density $d(S,H_k)$ of the
    periodic solution $S$ in $H_k$ is $w(C)$ divided by~$2pk$, the
    number of vertices in~$P$. Thus, 
    $$d(S,H_k) = w(C)/(2pk) = (w(C)/p)(1/2k) = \widetilde{w}(C)/2k.$$


    Now, given a periodic {\lds} solution $S$ of $H_k$, let us prove
    the existence of a cycle $C$ in $G_k$, as desired.  Suppose $S$
    has a pattern~$P$ with period $2p$.  Let us assume, without loss 
    of generality, that $P$ corresponds to a $2p$-bar $B$ of $H_k$ 
    whose columns are numbered $1, 2, \ldots, 2p$. 
    Let $B'$ be the $(2p+4)$-bar of $H_k$ whose
    columns are numbered $-1, 0, 1, 2, \ldots, 2p+1, 2p+2$. Since $P$
    is a pattern, the columns~$-1$ and~$0$ (resp. columns 1 and 2) of
    $P$ are identical to the columns $2p-1$ and $2p$ (resp. $2p+1$ and
    $2p+2$). Thus, $G_k$ has a cycle
    $C = (v_{B_1}, v_{B_2}, \ldots, v_{B_{p}})$, where $B_1$ is the
    barcode of the first 4-bar of $B'$ (that starts at column $-1$);
    and for $j=2,\ldots,p$ we have that $B_j$ is the barcode of the
    4-bar of $B'$ that starts two columns after the previous barcode
    $B_{j-1}$. Based on the  analysis done in the previous case, we
    conclude that cycle $C$ generates the pattern $P$ and $w(C)$ is
    the number of vertices of $S$ in the pattern~$P$. Thus,
    $d(S,H_k)= w(C)/2pk= \widetilde{w}(C)/2k$.
    Since $G_k$ has a positive outdegree at every vertex, it follows
    that $G_k$ has a cycle. This fact, together with the previous
    results, implies that a minimum mean cycle in $G_k$ yields a
    minimum-density periodic {\lds} of $H_k$.

    It  remains to prove that $H_k$ always admits an optimal {\lds}
    solution that is periodic. We omit this proof, and refer the
    reader to Jiang~\cite{Jiang18}, who proved that the square grid
    $S_k$ always admits an optimal identifying code that is
    periodic. The proof for {\lds} in $H_k$ can be done analogously,
    as $H_k$ has the Slow Growth property, and a superset of an {\lds}
    is also an {\lds}. Thus, $H_k$ always has an optimal {\lds}
    solution that is periodic, and whose period has length at most the
    order of $G_k$. Therefore, such a solution corresponds to a
    minimum mean cycle in $G_k$.
\end{proof}

\subsection{Our implementation of the algorithm
  \texorpdfstring{$A_k$}{Ak} and optimal  solutions for \texorpdfstring{$H_2$}{H2},
    \texorpdfstring{$H_3$}{H3}, \texorpdfstring{$H_4$}{H4},  \texorpdfstring{$H_5$}.}

  To find a minimum-density \lds{} in $H_k$, for $k\in\{2,3,4,5\}$, we
  implemented the algorithm $A_k$ in C++.
%
We used the LEMON library for the graph data
structure and a minimum mean cycle algorithm.  Specifically, rather
than using Karp's~\cite{Karp78} or Hartmann and
Orlin's~\cite{HartmannO93} algorithms, we used LEMON's implementation
of Howard's policy iteration algorithm.  Although originally designed
for finding optimal policies in \emph{Markov Decision Processes}, this
algorithm can be adapted to find minimum mean cycles in
digraphs~\cite{DasdanG98}.  While the best-known theoretical bounds
for Howard's algorithm are exponential, in practice it proved to be
significantly faster than the other two (polynomial-time) algorithms.

The code was compiled using g++ with the -O2 optimization flag and was executed 
on a high-performance computing (HPC) cluster. The executions used 20 cores of
an Intel(R) Xeon(R) CPU E7-2870 @ 2.40GHz with 480 GB of RAM available.

Table~\ref{tab:time_spent} shows the size of the configuration digraph
and the execution time for $k \in \{2, 3, 4, 5\}$.  The number of
vertices (resp. arcs) in $G_k$ is at most $2^{4k}$ (resp. $2^{6k}$).
The results obtained with our implementation are illustrated in
Figures~\ref{fig:H2-H3}, \ref{fig:H4} and~\ref{fig:H5}.

As the size of the digraph $G_k$ grows exponentially with $k$, with
the computational resource available to us we were able to find
optimal solutions for $H_k$ only when $k\leq 5$.  In the next section
we present an approach to find feasible solutions for $H_k$ for
larger values of $k$ (and comment on the special case in which
$k$ is a multiple of~$3$).

\begin{table}
    \centering
    \renewcommand{\arraystretch}{1.15}
    \begin{tabular}{|c|r|r|r|r|r|}
        \hline
        \multicolumn{3}{|c|}{Size of $G_k$} &
        \multicolumn{3}{c|}{Running Time} \\
        \hline
        \textbf{$k$} &
        \multicolumn{1}{c|}{Vertices} &
        \multicolumn{1}{c|}{Arcs} &
        \multicolumn{1}{c|}{Construction} &
        \multicolumn{1}{c|}{Howard's Alg.} &
        \multicolumn{1}{c|}{Total} \\
        \hline\hline
        2 & 181 & 242 & 0.0016\,s  & 0.0001\,s & 0.0018\,s \\
        \hline
        3 & 2581 & 104218 & 0.270\,s & 0.011\,s & 0.281\,s \\
        \hline
        4 & 36917 & $\approx 5.4 \times 10^6$ & 41.38\,s & 0.95\,s & 42.33\,s \\
        \hline
        5 & 528005 & $\approx 2.76 \times 10^8$& 8711.49\,s & 98.39\,s & 8809.88\,s \\
        \hline
    \end{tabular}
    \caption{Size of the digraph $G_k$ and running time of the algorithm $A_k$.}
    \label{tab:time_spent}
\end{table}

\begin{theorem}
  The locating-dominating sets shown in Figures~\ref{fig:H2-H3}, \ref{fig:H4}
  and~\ref{fig:H5} are of minimum density in their corresponding grids. The optimal
  densities are $d^*(H_2) = 3/8 = 0.375$, $d^*(H_3) = 1/3$, $d^*(H_4) = 11/32 = 0.34375$
  and $d^*(H_5) = 12/35 \approx 0.3428$. 
\end{theorem}


\begin{figure}[ht]
    \centering
    \begin{minipage}[b]{0.5\linewidth}
        \centering
            \scalebox{1}{\includegraphics{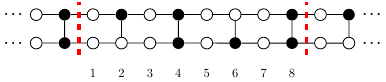}}
    \end{minipage}
    \begin{minipage}[b]{0.5\linewidth}
        \centering
            \scalebox{1}{\includegraphics{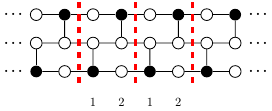}}
    \end{minipage}
    \caption{Optimal periodic solutions $\mathscr{S}_2$ and $\mathscr{S}_3$
            for the grids $H_2$ and $H_3$, with densities $d^*(H_2) = 3/8 = 0.375$ 
            and $d^*(H_3) = 1/3$.}
    \label{fig:H2-H3}
\end{figure}


\begin{figure}[ht]
    \centering
    \includegraphics{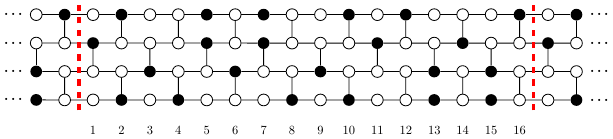}
    \caption{An optimal periodic solution $\mathscr{S}_4$ for the grid 
        $H_4$ with density $d^*(H_4) = 11/32 = 0.34375$.}
    \label{fig:H4}
\end{figure}

\begin{figure}[ht]
    \centering
    \includegraphics{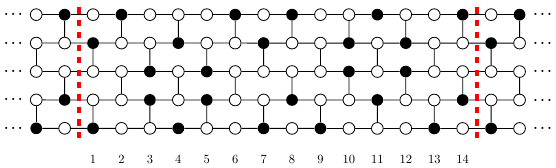}
    \caption{An optimal periodic solution $\mathscr{S}_5$ for the grid 
        $H_5$ with density $d^*(H_5) = 12/35 \approx 0.3428$.}
    \label{fig:H5}
\end{figure}

\section{An integer program to find feasible periodic locating-dominating sets in 
    \texorpdfstring{$H_k$}{Hk}}
\label{sec:ilp}

In this section, we consider the problem of finding, for a fixed integer 
$k\geq 6$, a periodic feasible \lds{} solution for $H_k$.  
For that, once $k$ is fixed, our approach is to fix an even integer $\ell \geq 6$,
and find an optimal \lds{} periodic solution on $H_k$ with period $\ell$ 
(which clearly exists).

To tackle this problem, we formulate it as a 0-1 integer linear
program (ILP, for short). Of course, an optimal solution for a fixed
$\ell$ provides only a feasible solution for $H_k$, and therefore an
upper bound for $d^*(H_k)$.  Our strategy is to consider different
values for $\ell$, and take an optimal solution with the smallest
density found. We note that as $\ell$ increases, the size of the ILP
may become prohibitively large (although its size is polynomial in $k$
and $\ell$). We comment on the values for $\ell$ that we considered
for~$k\in\{7,8\}$, guided by the information we have on a lower bound
for $d^*(H_k)$.

For a fixed pair $k, \ell$, let us denote by $\mathcal{P}_{k,\ell}$ the
ILP that finds a periodic solution for $H_k$ with period~$\ell$.
For that, let us consider an $\ell$-bar $B$ of the grid $H_k$, whose
rows are numbered $\{1, 2, \ldots, k\}$, and columns are numbered
$\{1, 2, \ldots, \ell\}$.  Thus, the vertices of $B$ are pairs
$(i,j)$ with $i\in [k]$ and $j\in [\ell]$. We adopt the convention
that the vertex $(1,1)$ is a vertex of degree~$2$ in $H_k$.
Our task now is to select vertices in $B$ to be part of a
solution that will be a periodic solution of $H_k$, and 
whose pattern is precisely $B$ together with these selected
vertices. 
    
Thus, we must ensure that the locating-dominating properties of the
constructed pattern hold not only when the vertices in $B$ are
tested with respect to their neighbors in $B$, but when we consider their
neighborhoods across the boundaries of $B$, when consecutive copies
of this pattern is repeated side by side to form a solution for~$H_k$.
  
To model this boundary interaction, we add to $B$ the set of edges
$\{(i, 1), (i, \ell)\}$ for each $i \in [k]$. That is, we connect the
vertices that belong to the first and last columns of $B$, and lie in
the same row of $B$.  This operation creates a subgraph $B^+$ which we
call a \emph{circular $\ell$-bar} with $k$ rows.  See in
Figure~\ref{fig:circular_bar} a circular~$6$-bar with $4$-rows, in
which the dashed edges are the ones added to $B$.

\begin{figure}[ht]
    \centering
    \includegraphics{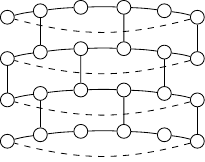}
    \caption{A circular $6$-bar with 4 rows.}
    \label{fig:circular_bar}
\end{figure}


Now, to find a periodic pattern, it suffices to work solely on the
circular $\ell$-bar with $k$ rows. To simplify notation, in the ILP
formulation, we refer only to vertices of $B$, but consider neighbors
with respect to~$B^+$; that is, column indices are interpreted cyclically within
$\{1, \ldots, \ell\}$, so that column $\ell$ is adjacent to column
$1$. In what follows, we make this more precise.

To write the constraints of $\mathcal{P}_{k,\ell}$, let us refer to
$S$ as the set of vertices of $B$ that we want to select. That is, for
each vertex $v \in V(B)$, we associate a binary variable $x_v$ and
impose constraints to guarantee that $x_v = 1$ if and only if
$v \in S$.

First, to impose that $S$ be a dominating set, we add the following
constraint: 
\[
    \sum_{u \in N[v]} x_u \geq 1\;\textnormal{ for each } v \in V(B). 
\]

To simplify notation, from now on, if $F$ is a set of vertices,
instead of $\sum_{v \in F}x_v$, we use the simplified notation
$x(F)$. Thus, the above inequality simplifies to  $x(N[v]) \geq 1$. 
Moreover, we note that, as this inequality is taken for each vertex $v\in V(B)$,
it is important that we take $N[v]$ instead of $N(v)$, so that if
$v$ is in $S$, this condition is automatically satisfied.

Next, to impose that $S$ must satisfy the locating property, that
is, each pair of distinct vertices~$u$ and~$v$ that are not in~$S$
must have distinct neighborhoods in $S$, we add the following
constraint:
\[
    x(N[u] \triangle N[v]) \geq 1\;\textrm{ for all } u,v\in V(B)
    \textrm{ such that dist}(u,v) = 2.
\]

We observe  that this constraint is added regardless of whether $u$ or
$v$ are elements of $S$, and again we consider closed neighborhoods
of $u$ and $v$. Note that, if $u$ or $v$ belongs to $S$, this
constraint is automatically satisfied, since $N[u] \triangle N[v]$
contains $u$ and $v$ (when they are at distance 2). The reader may
verify the validity of this inequality by checking the case in
which $u$ and $v$ belong (resp. do not belong) to the same row.

At this point, it is important to note that the inequality above is
\emph{not valid} when $\dist(u,v)=1$ and $u$ or $v$ belongs to~$S$. 
When both $u$ and $v$ do not belong to $S$, the inequality
is automatically satisfied (because of the first inequality above).
Thus, when $\dist(u,v)=1$, this inequality may not be included. 
We leave it to the reader to verify the above claims and note that there
is no need to consider the inequality corresponding to open
neighborhoods of $u$ and $v$ (as it would be redundant).

At last, we note that, when $\dist(u,v) > 2$, then $N[u]$ and $N[v]$
are disjoint, and hence $N[u] \triangle N[v] = N[u] \cup N[v]$, and,
once again, the domination constraints on $u$ and on $v$ already
guarantee that both $N[u]$ and $N[v]$ contain a vertex of~$S$.
Thus, pairs of vertices at distance greater than $2$ need not 
be considered to verify the locating property of~$S$. 
This justifies the inequality above.

Summarizing the previous discussions, only the first inequality, and
the second inequality restricted to pairs of vertices at distance~$2$
are necessary and sufficient.  Putting together the domination
constraint~(\ref{const:domination}), the locating constraints
(\ref{const:same_row}),~(\ref{const:prev_col})~and~(\ref{const:next_col}),
and the decision constraint~(\ref{const:decision}), we obtain the
0-1 integer linear program $\mathcal{P}_{k,\ell}$ described below.


As the column indices are in $[\ell]= \{1,2,\ldots, \ell\}$, we use congruence 
modulo $\ell$, taking the least positive residue system. 
To simplify notation in the ILP, for column index $j \in [\ell]$ we define
\[
    j^+ \coloneqq (j \bmod \ell) + 1, \qquad
    j^{++} \coloneqq ((j + 1) \bmod \ell) + 1, \qquad
    j^- \coloneqq ((j - 2) \bmod \ell) + 1,
\]
so that the column adjacencies are taken cyclically in $B^+$ (e.g., $\ell^+ = 1$, 
$\ell^{++} = 2$ and $1^- = \ell$). 
Similarly, we define $i^+ \coloneqq i + 1$.

\begin{samepage}
\begin{empheq}[left={(\mathcal{P}_{k,\ell})\quad}]{align}
    \textnormal{minimize} \quad & x\,\big(V(B)\big) & \\
    \textnormal{subject to} \quad 
    & x\,\big(N[v]\big) \geq 1 && \forall v \in V(B),  \label{const:domination} \\
    & x\,\big(N[u] \triangle N[v]\big) \geq 1, \; u=(i,j),\, v=(i,j^{++}) 
        && \forall i\in[k],\,\forall j\in[\ell],  \label{const:same_row} \\
    & x\,\big(N[u] \triangle N[v]\big) \geq 1, \; u=(i,j),\, v=(i^+,j^-) 
        && \forall i\in[k-1],\,\forall j\in[\ell],  \label{const:prev_col} \\
    & x\,\big(N[u] \triangle N[v]\big) \geq 1, \; u=(i,j),\, v=(i^+,j^+) 
        && \forall i\in[k-1],\,\forall j\in[\ell],  \label{const:next_col} \\
    & x_v \in \{0,1\}, && \forall v \in V(B). \label{const:decision} 
\end{empheq}
\end{samepage}

We implemented the integer program $\mathcal{P}_{k,\ell}$ in C++
using the \emph{Gurobi} optimizer.

For $k = 7$ (resp. $k=8$), we considered periods $\ell$ ranging over
the even integers from $6$ to $40$ (resp. from $6$ to $26$).
The experiments were run on a MacBook Air with an 8-core M1
  processor.  For $H_7$, the longest run ($7.08$ hours) occurred for
  $\ell = 40$, and for $H_8$, the longest run ($1.8$ hours) occurred
  for $\ell = 26$.  Among the periods that we considered, the best
solutions for $H_7$ and $H_8$ were found for $\ell = 22$ and
$\ell = 20$, yielding densities $26/77 \approx 0.3376$ and
$27/80 = 0.3375$, respectively.  
These solutions, named $\mathscr{S}_7$ and $\mathscr{S}_8$, are shown in 
Figures~\ref{fig:H7_pl} and~\ref{fig:H8_pl}. 
Both densities are within $1.3\%$ of the optimum, a fact that follows from the 
lower bound $1/3$, which we prove in the next section (see Corollary~\ref{coro:lower-bound}).


\begin{figure}[ht]
    \centering
    \includegraphics{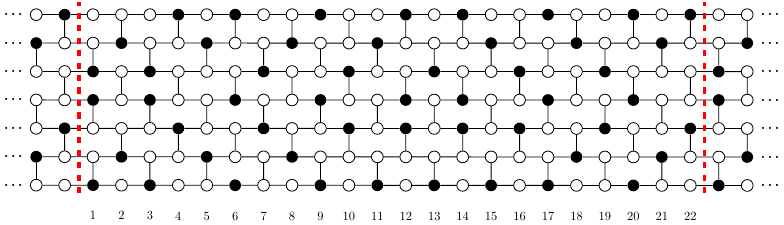}
    \caption{A periodic solution $\mathscr{S}_7$ for $H_7$ with density
    $d(\mathscr{S}_7, H_7) = 26/77 \approx 0.3376$.}
    \label{fig:H7_pl}
\end{figure}
\begin{figure}[ht]
    \centering
    \includegraphics{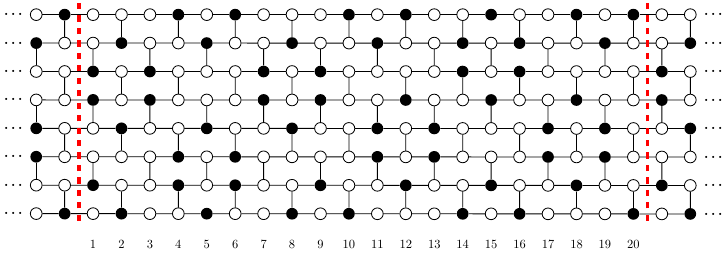}
    \caption{A periodic solution $\mathscr{S}_8$ for  $H_8$ with density
    $d(\mathscr{S}_8, H_8) = 27/80 = 0.3375$.}
    \label{fig:H8_pl}
\end{figure}

For $k = 6$, $k = 9$, and more generally, when $k$ is a multiple of $3$, optimal
solutions for $H_k$ can be constructed by combining solutions for $H_3$.
This construction, together with more general ones and the corresponding density
calculations, will be discussed in the next section.

%

\section{Explicit constructions of locating-dominating sets in hexagonal grids}
\label{sec:stack-hex-grid}

In Section~\ref{sec:min-density-hex-grid}, we showed an optimal
periodic solution for $H_3$ with period~$2$ and density~$1/3$.  This
optimal solution, denoted $\mathscr{S}_3$, is shown in
Figure~\ref{fig:H2-H3}.  Observe that the solution for $H_6$
illustrated in Figure~\ref{fig:H6} is also periodic, has period~$2$
and density~$1/3$. We discuss its optimality later.  This solution for
$H_6$ was obtained by combining two optimal solutions $\mathscr{S}_3$
for $H_3$.  The first~3 rows of~$H_6$ correspond to a copy of
$\mathscr{S}_3$, and the next~3 subsequent rows correspond to another
copy of $\mathscr{S}_3$.  In what follows, we define more formally the
construction of combining periodic solutions of any two grids (not
necessarily with the same period or the same number of rows) to obtain
a periodic solution for a larger grid.

\begin{figure}[ht]
    \centering
    \includegraphics{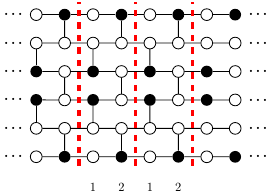}
    \caption{An optimal periodic solution $\mathscr{S}_6$ for $H_6$ with density $1/3$.}
    \label{fig:H6}
\end{figure}


\subsection{Combining periodic solutions of \texorpdfstring{$H_i$}{Hi} and 
\texorpdfstring{$H_j$}{Hj} to obtain a periodic solution 
for~\texorpdfstring{$H_{i+j}$}{H(i+j)}.}

We now formalize and generalize the construction illustrated above to
obtain a  solution for $H_6$.

For  $i,j\geq 2$, let $S_i$ and $S_j$ be periodic {\lds} solutions for the grids $H_i$
and $H_j$, with patterns $P_i$ and~$P_j$ of periods~$p$ and~$q$, respectively.  Let
$t \coloneqq \mathrm{lcm}(p, q)$ be the lowest common multiple of $p$ and $q$. Denote by
$H_i(S_i)$ and $H_j(S_j)$ the grid $H_i$ (resp.  $H_j$) together with the solution $S_i$
(resp.~$S_j$) represented on its vertices.  Considering the definition of $t$, it is
immediate that any set of~$t$ consecutive columns of $H_i(S_i)$ (resp. $H_j(S_j)$) defines
a pattern, say $P^*_i$ (resp.~$P^*_j$), of $H_i$ (resp.~$H_j$) with period $t$.

Let $S^*_i$ and $S^*_j$ be the periodic {\lds} solutions for $H_i$ and $H_j$ defined by
the patterns $P^*_i$ and $P^*_j$, respectively. Clearly, the densities of these patterns
are equal to the densities of the original patterns, namely,
$d(P^*_i, H_i) = d(P_i,H_i)= d(S_i, H_i)$ and $d(P^*_j, H_j) = d(P_j,H_j) = d(S_j, H_j)$.


\textbf{Combining solutions $S^*_i$ and $S^*_j$ to obtain a solution $S^*$ for $H_{i+j}$.}
For that, we simply consider that the first $i$ rows of $H_{i+j}$ are a copy of
$H_i(S^*_i)$, and the subsequent $j$ rows of $H_{i+j}$ are a copy of $H_j(S^*_j)$.  
More formally, the solution $S^*$ of $H_{i+j}$ restricted to the first $i$ rows of $H_{i+j}$
coincides with $S^*_i$, and $S^*$ restricted to the rows from $i+1$ to $i+j$ coincides
with $S^*_j$; that is, $S^* = S^*_i\cup S^*_j$. To refer to the solution obtained by the
procedure we just described, we write that $S^*:= S^*_i \oplus S^*_j$,
$P^*:= P^*_i \oplus P^*_j$ and $H_{i+j}(S^*) := H_{i}(S^*_i) \oplus H_{j}(S^*_j)$.

Observe that we did not impose any restriction regarding the initial columns of the
patterns $P^*_i$ and $P^*_j$ (when $i$ is odd, they will differ because of the structure
of the hexagonal grid, but this occurs naturally).
Note that $H_{i+j}$ has ``vertical'' edges connecting, in the same column, a vertex
in row~$i$ with a vertex in row~$i + 1$ -- vertices that had degree~$2$ in their
original grids $H_i$ and $H_j$. The proof of the next theorem shows that $S^*$
remains an \lds{} despite these new connections.

\begin{theorem}\label{thm:stack_grids} 
    Let $S_i$ and $S_j$ be periodic {\lds} solutions
    for the grids $H_i$ and $H_j$, $i,j\geq 2$, with patterns $P_i$ and $P_j$ of even
    periods~$p$ and $q$, respectively.  Let $P^*_i$ (resp. $P^*_j$) be patterns of period
    $t$, where $t= \mathrm{lcm}(p, q)$, obtained by taking any~$t$ consecutive columns of
    $H_i(S_i)$ (resp. $H_j(S_j)$).  Let~$S^*_i$ and~$S^*_j$ be the periodic solutions for
    $H_i$ and $H_j$ defined by the patterns $P^*_i$ and $P^*_j$, respectively. Let
    $P^*:= P^*_i \oplus P^*_j$ and $S^* := S^*_i\oplus S^*_j$ be as defined above.
    Then, $S^*$ is a periodic {\lds} solution for
    $H_{i+j}$ with pattern $P^*$ of period~$t$, whose density is
    \[
        d(S^*, H_{i+j}) =  d(P^*, H_{i+j}) = \frac{i\cdot d(S_i, H_i) + j\cdot d(S_j, H_j)}{i+j}.
    \]
\end{theorem}

\smallskip

\begin{proof}
    Let $S_i$, $S_j$, $S^*$, $P_i$, $P_j$, $P^*_i$, $P^*_j$ and $P^*$ as
    in the statement of the theorem.
    First, we prove that $S^*$ is a periodic \lds{} solution of $H_{i+j}$ with
    pattern $P^*$. 
    First, it is immediate that $S^*$ is a dominating set; so it suffices to 
    verify that $S^*$ satisfies the locating property in $H_{i+j}$.

    Considering the constraints of the ILP presented in the previous
    section, it suffices to verify that all pairs of vertices $(u,v)$
    at distance~2 in $H_{i+j}$ satisfy the inequality 
    $|(\tilde{N}[u]\triangle \tilde{N}[v])\cap S^*|\geq 1$, where
    $\tilde{N}(x)$ refers to the neighborhood of a vertex $x$ in
    $H_{i+j}$.

    In what follows, to simplify the writing, we say that a pair of vertices 
    $(u,v)$ of $H_{i+j}$ satisfies the \emph{locating constraint for $S^*$},
    if $|(\tilde{N}[u]\triangle \tilde{N}[v])\cap S^*|\geq 1$.
    Thus, $S^*$ satisfies the locating property if all pairs $(u,v)$ of vertices
    at distance~2 in $H_{i+j}$ satisfy the locating constraint for $S^*$.

    For pairs of vertices $(u,v)$ of $H_{i+j}$ at distance~$2$, that are not in
    row~$i$ or row~$i+1$, henceforth named \emph{border rows}, it is immediate 
    that they satisfy the locating constraint for $S^*$ (as they satisfy this 
    constraint for $S_i^*$ or $S_j^*$).
    So, we only have to verify whether $(u,v)$ satisfies the locating constraint 
    for $S^*$ when at least one of the vertices $u$ or $v$ is in a border row.

    \textbf{Case 1.} If only $u$ belongs to a border row, assume, without loss
    of generality, that $u$ is in the border row of $H_i$.
    Then,~$v$ also belongs to $H_i$. In this case, clearly the pair $(u,v)$
    satisfies the locating constraint for $S^*$.

    \smallskip

    \textbf{Case 2.} When both $u$ and $v$ belong to the same border row
    (either $i$ or $i+1$). Without loss of generality, consider that $u$
    and $v$ are vertices in row~$i$. Then $u$ and $v$ are vertices of
    $H_i$ with the same degree: either~$2$ or~$3$ (because they are at
    distance~$2$). In both cases, since, in $H_i$, the pair~$(u,v)$ 
    satisfies the locating constraint for $S_i^*$, in  $H_{i+j}$ it 
    satisfies the locating constraint for $S^*$.

    \smallskip

    \textbf{Case 3.}  When $u$ and $v$ belong to different border rows.
    By symmetry, it suffices to analyse the case in which~$u$ is in row~$i$ 
    and~$v$ is in row~$i+1$.
    Since~$u$ and $v$ are at distance $2$ in $H_{i+j}$, they have different 
    degrees in their original grids: if $u$ has degree~$2$ in $H_i$ then $v$ 
    has degree~$3$ in $H_{j}$, and vice-versa. 
    In this case, in $H_{i+j}$ the vertices $u$ and $v$ have a unique common neighbor.
    Since $\tilde{N}[u]\triangle \tilde{N}[v] \supset \{u,v\}$, if~$u$ or~$v$
    belongs to $S^*$, then the pair $(u,v)$ satisfies the locating constraint 
    for~$S^*$. 
    If both~$u$ and~$v$ do not belong to $S^*$, as one of them has degree~$2$ 
    in the corresponding grid, without loss of generality, suppose it is the 
    vertex~$u$,
    then, $u$ has a neighbor in $H_i$ that belongs to $S^*$ but does not belong 
    to $\tilde{N}[v]$ (because $\tilde{N}[v]$ contains only vertices of
    $H_j$).
    Thus, $(u,v)$ satisfies the locating constraint for~$S^*$.

    \smallskip

    From the analysis of the three cases above we conclude that $P^*$ is
    a pattern of an \lds{} of~$H_{i+j}$.
    Now, to conclude the proof of the theorem, it remains to calculate the
    density of the periodic solution $S^*$ of~$H_{i+j}$ defined by the
    pattern $P^*$ (that has  period~$t$).

    For a pattern $P$ of a periodic solution $S^*$, we denote by $|P|$
    the total number of vertices in $P$, and by $|P(S^*)|$ the
    number of vertices of $S^*$ in pattern $P$. Thus,

    \[
        d(P^*, H_{i+j}) = \frac{|P^*(S^*)|}{(i+j)\cdot t} = \frac{|P^*_i(S_i^*)| +
        |P^*_j(S_j^*)|}{(i+j)\cdot t}.
    \]

    \smallskip

    Since $|P^*_i(S_i^*)|= d(P^*_i, H_i) \cdot (i\cdot t) = d(S_i, H_i) \cdot (i \cdot t)$, and
    analogously, $|P^*_j(S_j^*)|= d(S_j, H_j) \cdot (j \cdot t)$, substituting these values in
    the previous equality, we obtain

    \[
        d(P^*, H_{i+j}) = \frac{i\cdot d(S_i, H_i) + j\cdot d(S_j,
        H_j)}{i+j}.
    \]

\end{proof}

In Figure~\ref{fig:H11_stack}, the reader may see that the combined
solution $\mathscr{S}_{11}$ for $H_{11}$ that is obtained by combining the
optimal solution~$\mathscr{S}_3$ for $H_3$ (of period~$2$) and the solution
$\mathscr{S}_8$ for $H_8$ (of period $20$ shown in Figure~\ref{fig:H8_pl}) has
period $t = \mathrm{lcm}(2, 20) = 20$. Note that any $20$ consecutive
columns of the combined solution for $H_{11}$ defines a pattern of
this solution.

\begin{figure}[ht]
    \centering
    \includegraphics{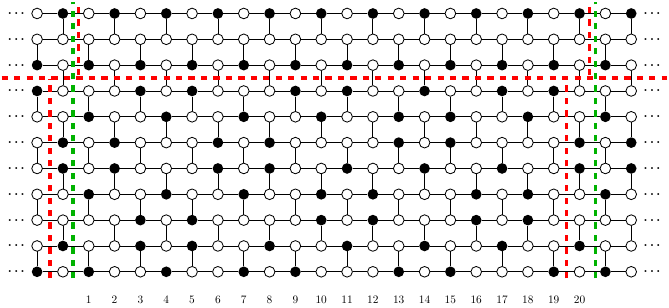}
    \caption{A periodic {\lds} solution $\mathscr{S}_{11} = \mathscr{S}_3 \oplus \mathscr{S}_8$ 
        for $H_{11}$ with period~$20$ and density 
        $d(\mathscr{S}_{11},H_{11}) = 37/110 \approx 0.3363$.}
    \label{fig:H11_stack}
\end{figure}

\bigskip

At this point it is important to mention the following result obtained
by Honkala and Laihonen~\cite{HonkalaL06} on the infinite hexagonal
grid $\mathcal{G}_H$. This result has many consequences in our
further studies.

\medskip

\begin{theorem} [Honkala and Laihonen~\cite{HonkalaL06}] \label{thm:min_infty_hex_grid}
    The minimum density of a locating-dominating set on the infinite hexagonal 
    grid $\mathcal{G}_H$ is $d^*(\mathcal{G}_H) = 1/3$.
\end{theorem}


\smallskip

\begin{corollary} \label{coro:lower-bound}
    The following results hold for $H_k$.
    \begin{enumerate}[label=\alabel]
        \item For all $k\geq 1$, we have $d^*(H_k) \geq 1/3$. 
        \item For all $m\geq 1$, we have $d^*(H_{3m}) = 1/3$.
    \end{enumerate}
\end{corollary}

\begin{proof}
  (a) We have proved that $d^*(H_3)= 1/3$ and $d^*(H_2)= 3/8 = 0.375$.
  Moreover, it is known that $d^*(H_1)= 2/5$~\cite{BertrandCHL04}. If
  for a fixed $k\geq 4$, a grid $H_k$ admitted a periodic {\lds}
  solution with density smaller than~$1/3$, then by
  Theorem~\ref{thm:stack_grids} the combination of infinitely many
  periodic solutions for $H_k$ would generate an \lds{} of
  $\mathcal{G}_H$ with density smaller than $1/3$, contradicting
  Theorem~\ref{thm:min_infty_hex_grid}.

(b) Let $m\geq 2$. By Theorem~\ref{thm:stack_grids}, we have
  that  $H_{3m}$ can be obtained by starting with the optimal
  solution $\mathscr{S}_3$ (for $H_3$) and combining $m-1$ times the partially
obtained solutions with another optimal solution
$\mathscr{S}_3$. 
Since $d^*(H_3) = 1/3$, by Theorem~\ref{thm:stack_grids}, we obtain a solution
for $H_{3m}$ with density $1/3$. 
By the previous result (a), we conclude that $d^*(H_{3m}) = 1/3$.
\end{proof}

\subsection{Constructions of feasible periodic solutions for
  \texorpdfstring{$H_k$}{Hk}, for all \texorpdfstring{$k\geq 10$}{k >=
    10}.}

Since Theorem~\ref{thm:stack_grids} provides a systematic method to
construct locating-dominating sets for larger grids, using the
periodic solutions we know for $H_k$ when
$k \in \{3, \ldots, 8\} \cup\{3m: m\geq 2\}$, we are able to construct
periodic solutions for $H_k$, for all $k\geq 10$. As we may combine
known solutions in many different ways, the natural question that
arises is: which combinations should we take to get better solutions?

Considering the solutions that we obtained for $H_k$,
$k \in \{2, \ldots, 8\}$, we observe that the optimal density is achieved when 
$k$ is a multiple of $3$. 
Moreover, for $k=7$ and $k=8$ the solutions $\mathscr{S}_7$ and $\mathscr{S}_8$ 
that we obtained (solving ILPs) are quasi-optimal (or possibly optimal): 
their densities are the closest ones to the lower bound of $1/3$. 
We verified that the density of the solution $\mathscr{S}_7$ 
(resp. $\mathscr{S}_8$) is better than the density obtained by the combined solutions
$\mathscr{S}_2\oplus \mathscr{S}_5$ or $\mathscr{S}_3\oplus \mathscr{S}_4$ 
(resp. $\mathscr{S}_2\oplus \mathscr{S}_6$ or $\mathscr{S}_3\oplus \mathscr{S}_5$).

Thus, to construct locating-dominating sets for $H_k$, $k\geq 10$, when $k$ is not a
multiple of $3$, our strategy is to combine the largest possible number of optimal
solutions $\mathscr{S}_3$ for $H_3$ with the solution $\mathscr{S}_7$ or $\mathscr{S}_8$,
depending on whether $ k=1 \bmod 3$ or $k=2\bmod 3$.

Concretely, when $k= 3m+1$, $m \geq 3$, since $3m + 1 = 3(m - 2) + 7$,
we combine $m - 2$ copies of the optimal solution $\mathscr{S}_3$ for $H_3$
(Figure~\ref{fig:H2-H3}) to obtain an optimal solution for
$H_{3(m-2)}$, and then we combine this optimal solution with a copy of
the solution $\mathscr{S}_7$ of $H_7$ (Figure~\ref{fig:H7_pl}).
By Theorem~\ref{thm:stack_grids}, this combination yields an {\lds} for
$H_{3m+1}$ with density $(11m + 4)/ (33m +11)$, obtained with the
formula given in Theorem~\ref{thm:stack_grids}, taking $i=3(m-2)$,
$j=7$, $d(\mathscr{S}_3, H_3) = 1/3$ and $d(\mathscr{S}_7,H_7) = 26/77$.

Similarly, when $k= 3m+2$, $m \geq 3$, since $3m+2 =3 (m-2) + 8$, we
proceed analogously to the previous case, use the solution $\mathscr{S}_8$ for
$H_8$ (Figure~\ref{fig:H8_pl}) with density $27/80$ and obtain a
solution for $H_{3m+2}$ with density $(10m + 7)/(30m +20)$.

Note that the combined solution when $k= 3m+1$ (resp. $k=3m+2$),
$m\geq 3$, has period $22$ (resp.~$20$). Since $\mathscr{S}_3$ has
period~$2$, each of these combined solutions has a very short
description.

We summarize the results of this section in the following theorem.

\begin{theorem}
\label{thm:upperbounds}
Let $H_k$ be the infinite hexagonal grid with $k$ rows. For each
$k\geq 1$, the following are the optimal values or the best upper
bounds known for  $d^*(H_k)$.

\begin{enumerate}[label=\alabel]
    \item $d^*(H_1) = 2/5$, $d^*(H_2) = 3/8 = 0.375$, $d^*(H_4) = 11/32 = 0.34375$,
        $d^*(H_5) = 12/35 \approx 0.3428$, $d^*(H_7) \leq  d(\mathscr{S}_7,H_7) = 26/77$, and
        $d^*(H_8) \leq d(\mathscr{S}_8,H_8) = 27/80$.

    \item For $k = 3m$, $m \geq 1$, we have that $d^*(H_k) = d^*(H_{3m})= 1/3$.

    \item \label{item:3m+1} For $k = 3m + 1$, $m \geq 3$, we have that
        $d^*(H_k) = d^*(H_{3m+1}) \leq (11m + 4)/ (33m + 11) $.  
        \rm{(}Solution obtained using combinations based on the equality $3m + 1 = 3(m - 2) + 7$.\rm{)}
        
    \item \label{item:3m+2} For $k = 3m + 2$, $m \geq 3$, we have that
        $d^*(H_k) = d^*(H_{3m+2}) \leq    (10m + 7)/ (30m +20)$  $\quad \qquad$ 
        (Solution obtained using combinations based on the equality  $3m + 2 = 3(m - 2) + 8$.)
\end{enumerate}
\end{theorem}

\medskip

    We conclude this section observing that the solutions $\mathscr{S}_7$ for $H_7$ 
    and $\mathscr{S}_8$ for $H_8$ are within $1.3\%$ of the optimum, and for all 
    the other cases the solutions for the corresponding grid $H_k$ (given by the 
    above theorem) are either optimal or within $1\%$ of the optimum. 
    Moreover, when~$k$ tends to infinity, the corresponding densities tend to $1/3$.


For $k < 10$, an optimality proof is missing only when $k = 7$ and $k = 8$. 
Therefore, an extra effort was made for these two cases. First, we used our 
ILP formulation to search over a wide range of periods, and the solutions 
$\mathscr{S}_7$ and $\mathscr{S}_8$ were the best we could find. Second, 
we used the configuration digraph approach restricted to low-density barcodes, 
and both of these solutions were recovered. Although neither approach certifies 
optimality, these computations provide strong evidence that the densities we 
obtained for $H_7$ and $H_8$ may be optimal.

With these solutions, we obtained high-quality, concisely described solutions 
for~$H_{3m+1}$ and~$H_{3m+2}$ as provided by~\ref{item:3m+1} and~\ref{item:3m+2} 
of Theorem~\ref{thm:upperbounds}. Driven by mathematical curiosity to improve 
these bounds, and by an analysis of the density patterns of these solutions 
as a function of $k$, we arrived at the following conjecture:

\begin{conjecture}\label{conj:exact-density}
    For every $k\ge1$,
    \[
        d^*(H_k)=
        \begin{cases}
            1/3,               & k \equiv 0 \pmod 3,\\
            1/3 + 1/(3k(k+4)), & k \equiv 1 \pmod 3,\\
            1/3 + 1/(3k(k+2)), & k \equiv 2 \pmod 3.
        \end{cases}
    \]
\end{conjecture}

The conjecture correctly reproduces all optimal values established in this paper, 
as well as the best solutions found computationally for $H_7$ and $H_8$. As 
an additional validation, we computed an optimum candidate for $H_{10}$ 
using both the ILP formulation and the reduced configuration digraph approach. 
Both approaches produced a periodic solution of density $47/140$, matching 
exactly the value predicted by the conjecture. We were unable to perform 
analogous computations for larger values of~$k$ because the size of the 
configuration digraph grows exponentially with~$k$.

We believe that proving or disproving this conjecture is a challenging open problem.



\section{Concluding remarks}

Our investigation into the {\MinDen} {\lds} problem on infinite
hexagonal grids $H_k$ has yielded tight bounds. Specifically, we
established solutions for all $k$ within $1.3\%$ of the optimum,
showing that they admit a concise description. The following statement
summarizes our core structural finding: for each $k \geq 10$, we can
specify exactly which two solutions from
$\{\mathscr{S}_3, \mathscr{S}_7, \mathscr{S}_8\}$ must be combined ---
in a straightforward manner --- to construct a concise solution
for~$H_k$ that is within $1\%$ of the optimum.

A preliminary version of this work, containing only a subset of the
current results and omitting proofs, appeared as an extended abstract
in~\cite{GomesW25Procedia}, a volume dedicated to LAGOS 2025.
Crucially, the results presented in Section~\ref{sec:ilp} regarding
the ILP formulation are substantially different from those in the
preliminary version.

The literature contains numerous results on the minimum density of
well-studied codes and sets in both infinite regular grids and
infinite strips of finite width. For interested readers, a highly
valuable resource is the comprehensive and continuously updated online
bibliography on locating-dominating sets, identifying codes, and
related topics for finite and infinite graphs, currently maintained by
Jean~\cite{Jean24} as a continuation of the foundational work by
Lobstein.

\bibliographystyle{abbrv_networks}
\bibliography{lds-Hk}

\medskip

\hrule

\end{document}